\documentclass[12pt]{amsart}
\usepackage[parfill]{parskip}
\usepackage{tikz-cd}
\usepackage{amsmath}
\usepackage{geometry}
\usepackage{latexsym}
\usepackage{amsthm}
\usepackage{amssymb}
\usepackage{amscd}
\usepackage{bm}
\usepackage[font=small,labelfont=bf]{caption}
\usepackage{graphicx}
\usepackage{epsfig}
\usepackage{epstopdf}
\usepackage{placeins}
\usepackage{mathtools}
\usepackage{subfigure}

\DeclarePairedDelimiter{\ceil}{\lceil}{\rceil}
\DeclarePairedDelimiter{\floor}{\lfloor}{\rfloor}
\providecommand{\N}{\mathbb{N}}
\providecommand{\R}{\mathbb{R}}

\providecommand{\vol}{\mbox{vol}}

\providecommand{\abs}[1]{\left\vert#1\right\vert}

\providecommand{\set}[1]{\left\{#1\right\}}

\providecommand{\floor}[1]{\lfloor\{#1\rfloor\}}

\providecommand{\paren}[1]{\left( #1 \right)}

\providecommand{\brac}[1]{\left [ #1 \right]}
\providecommand{\set}[1]{\left { #1 \right}}

\newcommand{\s}[1]{\begin{equation*} \begin{split} #1 \end{split} \end{equation*}}

\providecommand{\eqref}[1]{\left(\ref{#1}\right)}

\newtheorem{theorem}{Theorem}
\newtheorem{Lemma}{Lemma}

\newtheorem{Definition}{Definition}
\newtheorem{Proposition}{Proposition}

\DeclareMathOperator{\rank}{rank}

\newcommand{\PH}{\mathit{PH}}
\calclayout

\begin{document}
\title[Weighted Persistent Homology Sums of Random \v{C}ech Complexes]{Weighted Persistent Homology Sums \\ of Random \v{C}ech Complexes}
\author{Benjamin Schweinhart}
\date{July 2018}

\begin{abstract}
We study the asymptotic behavior of random variables of the form 
\[E_{\alpha}^i\paren{x_1,\ldots,x_n}=\sum_{\left(b,d\right)\in \PH_i\paren{x_1,\ldots,x_n}} \paren{d-b}^{\alpha}\]
where $\set{x_j}_{j\in\N}$ are i.i.d. samples from a probability measure on a triangulable metric space, and $\PH_i\paren{x_1,\ldots,x_n}$ denotes the $i$-dimensional reduced persistent homology of the \v{C}ech complex of $\set{x_1,\ldots,x_n}.$ These quantities are a higher-dimensional generalization of the $\alpha$-weighted sum of a minimal spanning tree; we seek to prove analogues of the theorems of Steele~\cite{1988steele} and Aldous and Steele~\cite{1992aldous} in this context.

As a special case of our main theorem, we show that if $\set{x_j}_{j\in\N}$ are distributed independently and uniformly on the $m$-dimensional Euclidean sphere, $\alpha<m,$ and $0\leq i <n,$ then there are real numbers $\gamma$ and $\Gamma$ so that 
\[ \gamma \leq \lim_{n\rightarrow\infty} n^{-\frac{m-\alpha}{m}}  E_i^{\alpha}\paren{x_1,\ldots,x_n} \leq \Gamma\]
in probability. More generally, we prove results about the asymptotics of the expectation of $E_\alpha^i$ for points sampled from a locally bounded probability measure on a space that is the bi-Lipschitz image of an $m-$dimensional Euclidean simplicial complex.
\end{abstract}
\maketitle

\section{Introduction}

We are interested in random variables of the form
\[E_{\alpha}^i\paren{x_1,\ldots,x_n}=\sum_{\left(b,d\right)\in \PH_i\paren{x_1,\ldots,x_n}}  \paren{d-b}^{\alpha}\]
where $\set{x_j}_{j\in\N}$ are independent samples drawn from a probability measure on a triangulable metric space, and $\PH_i\paren{x_1,\ldots,x_n}$ denotes the $i$-dimensional reduced persistent homology of the \v{C}ech complex of $\set{x_1,\ldots,x_n}.$ The special case $i=0$ is, under a different guise, already the subject of an expansive literature in probabilistic combinatorics; $E_\alpha^0\paren{\textbf{x}}$ gives the $\alpha$-weight of the minimal spanning tree on a finite subset of a metric space $\textbf{x},$ $T\paren{\textbf{x}}:$
\[E_{\alpha}^0\paren{\textbf{x}}=2^{-\alpha}\sum_{e\in T\paren{\textbf{x}}}\abs{e}^\alpha\]

In 1988, Steele~\cite{1988steele} showed the following:

\begin{theorem}[Steele]
Let $\mu$ is a compactly supported probability distribution on $\R^m,$ and let $\set{x_n}_{n\in\N}$ be i.i.d. samples from $\mu.$ If $\alpha<m,$
\[\lim_{n\rightarrow\infty} n^{-\frac{m-\alpha}{m}} E_\alpha^0\paren{x_1,\ldots,x_n} \rightarrow c\paren{\alpha,m}\int_{\R^d}f\paren{x}^{\paren{m-\alpha}/m}\]
with probability one, where $f\paren{x}$ is the probability density of the absolutely continuous part of $\mu,$ and $c\paren{\alpha,m}$ is a positive constant that depends only on $\alpha$ and $m.$
\end{theorem}

In 1992, Aldous and Steele~\cite{1992aldous} showed that if $\set{x_i}_{i\in\N}$ sampled independently from the uniform distribution on the unit cube in $\R^m,$ then
\[\lim_{n\rightarrow\infty} E_\alpha^m\paren{x_1,\ldots,x_n} \rightarrow c\paren{d,d}\]
in the $L^2$ sense. Under the same hypotheses, Kesten and Lee proved the following central limit theorem in 1996~\cite{1996kesten}:
\[\frac{E_{\alpha}^{0}\paren{X_1,\ldots,X_n}-\mathbb{E}\paren{E_{\alpha}^{0}\paren{X_1,\ldots,X_n}}}{n^{m-2\alpha}{2d}}\rightarrow N\paren{0,\sigma^2_{\alpha,d}}\]
in distribution, for any $\alpha>0.$ Here, we take the first step toward a higher-dimensional generalization of these celebrated results. 

Another special case of $E_{\alpha}^i\paren{\textbf{x}}$ --- $\alpha=1$ --- gives the total lifetime persistence of $\textbf{x}.$  Random variables of the form $E_{1}^i\paren{\textbf{x}}$ have been investigated by  Hiraoka and Shirai~\cite{2017Hiraoka} in the context of Linial---Meshulam processes. They showed that if $X$ is sampled from the $m$-Linial---Meshulam process then 
\[\mathbb{E}\paren{E_1^{m-1}\paren{X}} \in O\paren{n^{m-1}}\]
which is a higher-dimensional generalization of Frieze's $\zeta\paren{3}$-theorem for Erd\'{o}s---R\'{e}nyi random graphs~\cite{1985frieze}. Also, Adams et al.~\cite{2018adams} studied the behavior of the lifetime persistence of random measures on Euclidean space, performing computational experiments and conjecturing the existence of a limit function capturing finer properties of the persistent homology. 

The properties of $E_\alpha^i\paren{\textbf{x}}$ for general $i$ and $n$ have until now, as far as we know, not been studied in a probabilistic context (see the note at the end of the introduction). However, some work has been done in the extremal context. In 2010, Cohen-Steiner et al.~\cite{2010cohensteiner} showed that if $M$ is the bi-Lipschitz image of an $m$-dimensional simplicial complex and $\alpha>m,$ then $E_i^{\alpha}\paren{X}$ is uniformly bounded for $X\subset M.$ We use their results to prove the upper bounds in Section~\ref{sec_upper_bound}. Furthermore, in our previous paper~\cite{2018schweinhart} we related the upper box dimension of a subset $X$ of a metric space to the behavior of  $E_{\alpha}^i\paren{Y}$ for extremal subsets $Y\subset X.$ We will say more about the relation of this to the present work in Section~\ref{sec_ph_dim}.

\subsection{Our Results}
The following are special cases of our main theorem:
\begin{theorem}
\label{thm_sphere}
Let $\set{x_j}_{j\in\N}$ be be distributed independently and uniformly on the $S^n.$  If $\alpha<m,$ $0\leq i <n,$ and persistent homology is taken with respect to the intrinsic metric on $S^n,$
\[ \gamma \leq \lim_{n\rightarrow\infty} n^{-\frac{m-\alpha}{m}}  E_i^{\alpha}\paren{x_1,\ldots,x_n} \leq \Gamma\]
in probability, where $\gamma$ and $\Gamma$ are constants that depend on $\mu$ and $\alpha.$ 

Furthermore, there exists a $D\in\R$ so that
\[\lim_{n\rightarrow\infty}\frac{1}{\log\paren{n}}E_i^{m}\paren{x_1,\ldots,x_n} \leq D\]
in probability.
\end{theorem}

\begin{theorem}
\label{thm_ball}
Let $\set{x_j}_{j\in\N}$ be be distributed independently and uniformly on an $m$-dimensional Euclidean ball.  If $\alpha<m,$ $0\leq i <n,$
\[ \gamma \leq \lim_{n\rightarrow\infty} n^{-\frac{m-\alpha}{m}}  \mathbb{E}\paren{E_i^{\alpha}\paren{x_1,\ldots,x_n}} \leq \Gamma\]
where $\gamma$ and $\Gamma$ are constants that depend on $\mu$ and $\alpha.$ In fact, the lower bound holds in probability. 

Furthermore, there exists a $D\in\R$ so that
\[\lim_{n\rightarrow\infty}\frac{1}{\log\paren{n}}\mathbb{E}\paren{E_i^{m}\paren{x_1,\ldots,x_n}} \leq D\]
\end{theorem}

We show a stronger result for compactly supported probability measures on $\R^2$ that are locally bounded:
\begin{Definition}
A probability measure $\mu$ on $\R^m$ is \textbf{locally bounded} if there is a $A\subset \R^m$ with positive volume and real numbers $a_1 \geq a_0 >0$ so that 
\[a_0\,\vol\paren{B} \leq \mu\paren{B}\leq a_1\,\vol\paren{B}\]
for all Borel sets $B\subset A.$
\end{Definition}

\begin{theorem}
\label{thm_R2}
Let $\mu$ is a compactly supported, locally bounded probability measure on $\R^2,$ and let $\set{x_n}_{n\in\N}$ be i.i.d. samples from $\mu.$ If $\alpha<m,$ 
\[ \gamma \leq \lim_{n\rightarrow\infty} n^{-\frac{m-\alpha}{m}}  E_\alpha^{1}\paren{x_1,\ldots,x_n} \leq \Gamma\]
in probability. In fact, the upper bound holds with probability one.

Furthermore, there exists a constant $D$ so that 
\[\lim_{n\rightarrow\infty}\frac{1}{\log\paren{n}}E_2^{1}\paren{x_1,\ldots,x_n} \leq D\]
with probability one
\end{theorem}

More generally, we prove results for locally bounded probability measures on spaces that are the bi-Lipschitz image of a compact, $m$-dimensional Euclidean simplicial complex:
\begin{Definition}
\label{defn_locally_bdd}
Let $M$ be the bi-Lipschitz image of a compact $m$-dimensional Euclidean simplicial complex $\Delta_M$ under a map $\phi_M.$ A probability measure $\mu$ on $M$ is \textbf{locally bounded} if there exists a subset $A\subset \Delta_M$ with positive $m$-dimensional volume, and real numbers $a_1\geq a_0 >0$ so that
 \[ a_0 \frac{\vol\paren{B}}{\vol\paren{\Delta_M}}\leq \mu\paren{\phi_M\paren{B}} \leq  a_1 \frac{\vol\paren{B}}{\vol\paren{\Delta_M}}\]
for all Borel sets $B\subseteq A.$
\end{Definition}
For example, a the uniform measure on a $m$-dimensional Riemannian manifold is locally bounded, as is any measure that is locally bounded with respect to the Riemannian volume.  

 While there exist metric spaces $M$ with point sets $\set{x_j}_{j\in\N}$ so that 
\[\abs{\PH_i\paren{x_1,\ldots,x_n}}\neq O\paren{n}\]
this is thought to be somewhat pathological behavior~\cite{2018schweinhart}.
\begin{Definition}
A probability measure $\mu$ on a triangulable metric space has \textbf{linear $\PH_i$ expectation} if 
\[\mathbb{E}\paren{\abs{\PH_i\paren{\set{x_1,\ldots,x_n}}}}\in O\paren{n}\]
Similarly, $\mu$ has \textbf{linear $\PH_i$ variance} if  
\[\mathbb{E}\paren{\paren{\abs{\PH_i\paren{\set{x_1,\ldots,x_n}}}-\mathbb{E}\paren{\abs{\PH_i\paren{\set{x_1,\ldots,x_n}}}}}^2}\in O\paren{n}\]

\end{Definition}
For example, the uniform measure on a Euclidean ball~\cite{1991dwyer} and any positive, continuous probability density on the Euclidean n-sphere~\cite{2014stemeseder} has linear $\PH_i$ expectation. It is more difficult to prove that a probability measure has linear $\PH_i$ variance. As far as we are aware, this is only known for probability measures on $\R^2$ and the uniform measure on the $n$-dimensional Euclidean sphere~\cite{2014stemeseder} (see Equation~\ref{eqn_delaunay} and Proposition~\ref{corollary_linear}).

\begin{theorem}
\label{thm_main}
Let $M$ be the bi-Lipschitz image of an $m$-dimensional Euclidean simplicial complex, and $0\leq i< m.$ If $\mu$ is a locally bounded probability measure on $M,$ there are real numbers $0<\gamma<\Gamma$ so that 
\[\gamma n^{\frac{m-\alpha}{m}} \leq \mathbb{E}\paren{E_\alpha^i\paren{x_1,\ldots,x_n}}\leq  \Gamma\,\mathbb{E}\paren{\abs{\PH_i\paren{\set{x_1,\ldots,x_n}}}}^{\frac{m-\alpha}{m}}\]
for all sufficiently large $n.$  In particular, if $\mu$ has linear $\PH_i$ expectation, there is a real number $\Gamma_0$ so that 
\[\gamma \leq \lim_{n\rightarrow\infty}  n^{-\frac{m-\alpha}{m}}\,\mathbb{E}\paren{E_\alpha^i\paren{x_1,\ldots,x_n}}\leq  \Gamma_0\]
The lower bound holds in probability, and the upper bound does if $\mu$ has linear $\PH_i$ variance.

Furthermore, there exists a real number $D$ so that 
\[\mathbb{E}\paren{E_n^i\paren{x_1,\ldots,x_n}}\leq  D\,\log\paren{\mathbb{E}\paren{\abs{\PH_i\paren{x_1,\ldots,x_n}}}}\]
where analogously sharper statements hold if $\mu$ has linear $\PH_i$ expectation or variance.
\end{theorem}

We prove the upper bound in Proposition~\ref{prop_sharpUpper} and the lower bound in Proposition~\ref{prop_lower}.

After completion of this manuscript, we became aware that Divol and Polonik~\cite{2018divol} independently and concurrently proved a sharper result for the persistent homology of points sampled from  bounded, absolutely continuous probability densities on $\brac{0,1}^m.$ We believe this manuscript is still useful in that the proofs are largely self-contained, and the methods are applicable to other situations. In a later paper~\cite{2018schweinhart_c}, we use them to study the behavior of $E_\alpha^i\paren{x_1,\ldots,x_n}$ for i.i.d. points sampled from a measure supported on a set of fractional dimension.

\subsection{$\PH$-dimension}
\label{sec_ph_dim}

In~\cite{2018schweinhart}, we defined a family of persistent homology dimensions for a subset $X$ of a metric space $M$ in terms of the extremal behavior of $E_{\alpha}^i\paren{Y}$ for subsets $\textbf{x}$ of $X$:
\s{\text{dim}_{\PH}^i\paren{X}=\inf\set{\alpha : E_\alpha^i\paren{\textbf{x}} <C \;\forall \;\textbf{x}\subset X}}
That is, $E_\alpha^i\paren{\textbf{x}}$ is uniformly bounded for all $\alpha>\text{dim}_{\PH}^i\paren{X},$ but not for $\alpha<\text{dim}_{\PH}^i\paren{X}.$ Note that the persistent homology is taken with respect $M.$ Our results were the first rigorously relating persistent homology to a classically defined fractal dimension, the upper box dimension, but the definition is difficult to compute with in practice. Here, we define a similar notion of fractal dimension for measures on a metric space that may be more computable in practice:

\begin{Definition}
The $\PH_i$-dimension of a probability measure on a a triangulable metric space is 
\s{\text{dim}_{\PH}^i\paren{\mu}=\sup\set{\alpha : \limsup_{n\rightarrow\infty}\mathbb{E}\paren{E_\alpha^i\paren{x_1,\ldots,x_n}}=\infty}}
\end{Definition}

Clearly, $\text{dim}_{\PH}^i\paren{\mu}\leq \text{dim}_{\PH}^i\paren{\text{supp}\,\paren{\mu}}.$ As a corollary to our main theorem, we show:

\begin{theorem}
Let $M$ be the bi-Lipschitz image of a compact $m$-dimensional Euclidean simplicial complex, and $0\leq i< m.$ If $\mu$ is a locally bounded probability measure on $M,$ \[\text{dim}_{\PH}^i\paren{\mu}=m\]
\end{theorem}
\subsection{Persistent Homology}
If $X$ is a bounded subset of a triangulable metric space $M$, let $X_{\epsilon}$ denote the $\epsilon$-neighborhood of $X:$
\[X_{\epsilon}=\set{x\in M: d\paren{x,X}<\epsilon}\]
Also, let $H_i\paren{X}$ be the reduced homology of $X$, with coefficients in a field $k$. The \textbf{persistent homology} of $X$ is the product $\prod_{\epsilon>0} H_i\paren{X_\epsilon},$ together with the inclusion maps $i_{\epsilon_0,\epsilon_1}:H_i\paren{X_{\epsilon_0}}\rightarrow H_i\paren{X_{\epsilon_1}}$ for $\epsilon_0<\epsilon_1.$ The structure of persistent homology is captured by a set of intervals, which we refer to as $\PH_i\paren{X}$~\cite{2005zomorodian}. These intervals represent how the topology of $X_{\epsilon}$ changes as $\epsilon$ increases. Under certain finiteness hypotheses --- which are satisfied if $X$ is a finite point set --- $\PH_i\paren{X}$ is the unique set of intervals so that the rank of $i_{\epsilon_0,\epsilon_1}$ equals the number of intervals containing $\paren{\epsilon_0,\epsilon_1}$~\cite{2007cohensteiner}.

If $X$ is finite $\PH_i\paren{X}$ is the same as the persistent homology of the \v{C}ech complex of $X.$ Note that this depends on the ambient metric space. Here, if ``$\mu$ is a probability measure on $M$ and $\set{x_j}_{j\in\N}$ are sampled from $\mu$,'' then $\PH_i\paren{x_1,\ldots,x_n}$ is the persistent homology with ambient metric space $M.$ All questions we study here would also be interesting in the context of the Vietoris---Rips Complex. 

\subsection{Notation}
In the following, an $m$-space will be the bi-Lipschitz image of a compact $m$-dimensional Euclidean simplicial complex. Also, if the measure $\mu$ is obvious from the context, $\set{x_j}_{j\in\N}$ will denote a collection of independent random variables with common distribution $\mu.$ Also, $\textbf{x}_n$ will be shorthand for $\set{x_1,\ldots,x_n}$ and $\textbf{x}$ will denote a finite point set.

\section{Upper Bounds}
\label{sec_upper_bound}
Our strategy to prove an upper bound for the asymptotics of $E_\alpha^i\paren{\set{x_1,\ldots,x_n}}$ will be to bound the number and length of the persistent homology intervals in terms of the number of simplices in a triangulation of the ambient metric space. The approach is similar to that in our earlier paper~\cite{2018schweinhart}.

\subsection{Preliminaries}

We require the following result, which is proven by bounding the number of persistent homology intervals of a triangulable metric space of length greater than $\delta$ in terms of the number of simplices in a triangulation of mesh $\delta$:

\begin{Proposition}(Cohen-Steiner, Edelsbrunner, Harer, and Mileyko~\cite{2010cohensteiner})
\label{Proposition_CohenSteiner}
Let $M$ be an $m$-space. There exists a real number $C_0$ so that for any $0\leq i < m,$ $X\subseteq M,$ and $\delta>0,$ 
\[\abs{\set{\paren{b,d}\in\PH_i\paren{X}:d-b>\delta}} \leq C_0\,\delta^{-m}\]
\end{Proposition}
We use this result to bound $E_\alpha^i\paren{\textbf{x}}$ in terms of the number of $\PH_i$ intervals of $\textbf{x}$:

\begin{Lemma}
\label{lemma_upper}
Let $M$ be an $m$-space, $\alpha<m,$ and $i\in\mathbb{N}.$ There exists a real number $C_1>0$ so that 
\[E_\alpha^i\paren{X}\leq C_1 \abs{\PH_i\paren{X}}^{\frac{m-\alpha}{m}}\]
 for all $X\subseteq M.$ Furthermore, there exists a real number $D_1>0$ so that 
\[E_m^i\paren{X}\leq D_1 \log\paren{\abs{\PH_i\paren{X}}}\]
 for all $X\subseteq M.$ 
\end{Lemma}
\begin{proof}

Dilating $M$ by a factor $r$ multiplies $E_\alpha^i\paren{X}$ by $r^{\alpha},$ so we may assume without loss of generality that the diameter of $M$ is less than one. Let $n=\abs{\PH_i\paren{X}}$ and 
\[I_k=\set{\paren{b,d}\in\PH_i\paren{X}:\frac{1}{2^{k+1}} < d-b \leq \frac{1}{2^k}}\]
Also, let $C_0$ be as in Proposition~\ref{Proposition_CohenSteiner} so
\[\abs{I_k} \leq C_0 2^{m k}\]

The largest $C_0$ intervals of $\PH_i\paren{X}$ each have length less than or equal to $2^{0},$ the next largest $C_0 2^m$ intervals have length less than or equal to $2^{-1},$ and so on. It follows that if
\[l=\ceil[\Big]{\frac{\log_2\paren{2n/C_0}}{m}}\]
then
\[n\leq \sum_{k=0}^{l} C_0 2^{m k}\]
and
\[E_\alpha^i\paren{X} \leq \sum_{k=0}^{l} C_0 2^{m k} \paren{\frac{1}{2^k}}^\alpha\]
If $\alpha=m,$ the previous inequality becomes 
\[E_\alpha^i\paren{X} \leq C_0 l = O\paren{\log\paren{n}}\]
as desired.

Otherwise, if $\alpha<m,$
\begin{align*}
E_\alpha^i\paren{X} \leq & \\
&\;\; \sum_{k=0}^{l} C_0 2^{k\paren{m-\alpha}} \\
=&\;\; C_0 \frac{2^{\paren{m-\alpha}\paren{l+1}}-1}{2^{m-\alpha}-1}\\
\leq  &\;\; \frac{C_0}{2^{m-\alpha}-1} 2^{\paren{m-\alpha}\paren{l+1}}\\
\leq & \;\; \frac{C_0}{2^{m-\alpha}-1} 2^{\paren{\frac{\log_2\paren{2n/C_0}}{m}+2}\paren{m-\alpha}}\\
= & \;\; C_1 n^{\frac{m-\alpha}{m}}
\end{align*}
where $C_1=\frac{C_0 4^{m-\alpha}}{2^{m-\alpha}-1}.$
\end{proof}

\subsection{The Upper Bound}

The upper bound in our main theorem now follows immediately from Jensen's inequality, as the function $f\paren{x}=x^\frac{m-\alpha}{m}$ is concave for $0<\alpha\leq m:$
\begin{Proposition}
\label{prop_sharpUpper}
Let $M$ be an $m$-space, let $i$ be a natural number less than $m,$ and let $\mu$ be a locally bounded probability measure on $M.$ For all $0<\alpha<m$ there exists a real number $C>0$ so that 
\[\mathbb{E}\paren{E_\alpha^i\paren{x_1,\ldots,x_n}}\leq  C\,\mathbb{E}\paren{\abs{\PH_i\paren{x_1,\ldots,x_n}}}^{\frac{m-\alpha}{m}}\]
In particular, if $\mu$ has linear $\PH_i$ expectation and linear $\PH_i$ variance then there is a $C'>0$ so that
\[\lim_{n\rightarrow\infty} n^{-\frac{m-\alpha}{m}}E_\alpha^i\paren{x_1,\ldots,x_n}\leq  C'\]
in probability.

Furthermore, there exists a real number $D$ so that 
\[\mathbb{E}\paren{E_m^i\paren{x_1,\ldots,x_n}}\leq D\,\log\paren{\abs{\PH_i\paren{x_1,\ldots,x_n}}}\]
In particular, if $\mu$ has linear $\PH_i$ expectation and linear $\PH_i$ variance then there is a $D'>0$ so that
\[\lim_{n\rightarrow\infty} \frac{1}{\log\paren{n}} E_m^i\paren{x_1,\ldots,x_n}\leq  D'\]
in probability.

\end{Proposition}
\begin{proof}
Let $\textbf{x}$ be a finite subset of $B,$ and let $C_1$ be as in Lemma~\ref{lemma_upper}. If $\alpha<m,$
\begin{align*}
\mathbb{E}\paren{E_\alpha^i\paren{\textbf{x}}}\leq &\\
&\;\; \mathbb{E} \paren{C_1 \abs\PH_i\paren{\textbf{x}}^{\frac{m-\alpha}{m}}} &&\text{by Lemma~\ref{lemma_upper}}\\
\leq & \;\; C_1 \mathbb{E}\paren{\abs\PH_i\paren{\textbf{x}}}^{\frac{m-\alpha}{m}} &&\text{by Jensen's inequality}
\end{align*}
as desired. If $\mu$ has linear $\PH_i$ expectation and linear $\PH_i$-variance, Chebyshev's Inequality implies that
\[\lim_{n\rightarrow\infty} \abs{\PH_i\paren{x_1,\ldots,x_n}}/n \leq C_2 \]
in probability, for some $C_2>0,$ and the desired statement follows from Lemma~\ref{lemma_upper}.

The proof for the case $\alpha=m$ is similar.
\end{proof}

\subsection{Sharper Upper Bounds}

Our sharper upper bounds in Theorems~\ref{thm_sphere} and~\ref{thm_ball} follow from the fact that if $\set{x_1,\ldots,x_n}$ is a finite subset of $\R^m$ of $S^m$ in general position then
\begin{equation}
\label{eqn_delaunay}
\abs{\PH_i\paren{x_1,\ldots,x_n}}\leq \abs{DT\paren{x_1,\ldots,x_n}}
\end{equation}
where $DT\paren{x_1,\ldots,x_n}$ is the number of simplices of the Delaunay triangulation on $\set{x_1,\ldots,x_n}.$ In fact, the Alpha complex is a filtration on the simplices of the Delaunay triangulation that is homotopy equivalent to the $\epsilon$-neighborhood filtration of the points $\set{x_1,\ldots,x_n}$ ~\cite{2002edelsbrunner}. This construction is usually defined for points in Euclidean space, but easily extends to points on the $m$-sphere, in which case the Delaunay triangulation is the spherical convex hull of the points.

\begin{Proposition}
\label{corollary_linear}
If $B$ be a bounded subset of $\R^m$
\[E_\alpha^i\paren{x_1,\ldots,x_n} = O\paren{n^{\floor{\frac{m+1}{2}}\frac{m-\alpha}{m}}}\]
for any general position point set $\set{x_1,\ldots,x_n}$ contained in $B.$
\end{Proposition}
\begin{proof}
The Upper Bound Theorem~\cite{stanley1975} implies that if $X\subset \R^m$ then
\[\abs{\paren{DT}\paren{x_1,\ldots,x_n}}=O\paren{n^{\floor{\frac{m+1}{2}}}}\]

The desired statement follows immediately from Lemma~\ref{lemma_upper} and Equation~\ref{eqn_delaunay}
\end{proof}

\section{Lower Bounds}
Our strategy to prove lower bounds for the asymptotics of weighted $\PH$-sums is to study collections of sets whose persistent homology obeys a super-additivity property. We define certain ``cubical occupancy events'' giving rise to such collections, and prove that they occur with positive probability for sets of i.i.d. points drawn from a locally bounded probability measure on an $m$-space. We bootstrap these results by subdividing a subset of an $m$-dimensional cube into many small sub-cubes. This bootstrapping argument is similar to the one we used to prove a lower bound for $\PH_i$ dimension in our previous paper~\cite{2018schweinhart}.

In the following, fix $0\leq i <m.$

\subsection{Super-additivity for Persistent Homology}

Persistent homology does not in general obey a super-additivity property, but we can define a subclass of sets whose persistent homology does. If $X$ and $T$ are subsets of a triangulable metric space and $b<d,$ let $M_{X,T}\paren{b,d}$ be the rank of the homomorphism on homology induced by the inclusion
\[X_b\hookrightarrow X_d \hookrightarrow X_d \cup T_d\]
where $X_{\epsilon}$ denotes the $\epsilon$-neighborhood of $X.$ Note that 
\[M_{X,C}\paren{b,d}\leq N_X\paren{b,d}\]
where $N_X\paren{b,d}$ is the number of intervals of $\PH_i\paren{X}$ with birth times less than $b$ and death times greater than $d.$ We will show that if $C$ is an $m$-dimensional cube and $X\subset C,$ then quantities of the form $M_{X,\partial C}\paren{d,b}$ obey a super-additivity property.

\begin{Lemma}
\label{lemma_additivity}
\label{lemma_SA2}
Let $\set{C_1,\ldots, C_n}$ be $m$-dimensional cubes in $\R^m$ so that
\[ C_j\cap C_k\subset \partial C_j \qquad \forall\,j,k\in\set{1,\ldots,n}:j\neq k\]
If $X_j \subset C_j$ for $j=1,\ldots,n$
\[N_{\cup_j X_j}\paren{b,d} \geq M_{\cup_j X_j, \cup_j \partial C_j }\paren{b,d} \geq   \sum_{j=1}^{n} M_{X_j,\partial C_j}\paren{b,d}\]
for any $0\leq b <d.$ 
\end{Lemma}
\begin{proof}

Let $k\in \set{1,\ldots,n},$ $S=\cup_{j=1}^{k-1} X_j,$ $T=\cup_{j=1}^{n} \partial C_j,$ $X=X_k,$ and $C=C_k.$ See Figure~\ref{fig_LA}.

\begin{figure}
\centering
\includegraphics[width=.7\textwidth]{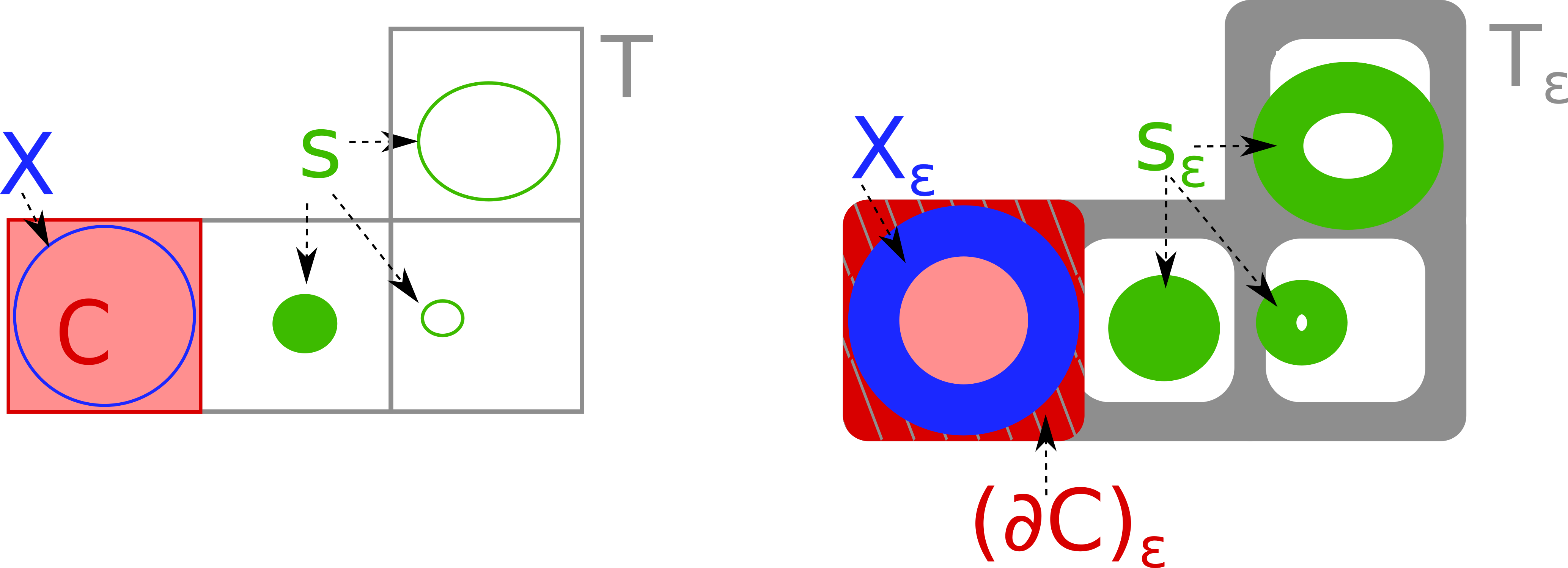}
\caption{The setup in the proof of Lemma~\ref{lemma_SA2}.}
\label{fig_LA}
\end{figure}

We consider the cases $i=m-1$ and $i<m-1$ separately. If $i=m-1,$ Alexander Duality implies that $N_{S}\paren{b,d}$ is the number of bounded components of the complement of $\paren{S_b}$ that intersect non-trivially with the complement of $S_d.$ Similarly,  $M_{X,C}\paren{b,d}$ is the number of bounded components of the complement of $X_b$ that intersect non-trivially with $\paren{X_d\cup \paren{\partial C }_{d}}^c.$ Note that all bounded components of $\paren{X_b}^c$ are contained within the interior of $C,$ because $C$ is convex and separates $\R^m$ into two components.

Let $Y$ be a component of the complement of $X_b$ that intersects non-trivially with $\paren{X_d\cup \paren{\partial C }_{d}}^c,$ and let $y\in Y\cap \paren{X_d \cup \paren{\partial C}_{d}}^c.$ $\partial C$ separates $\R^m$ into two components so
\[d\paren{y, S}\geq d\paren{y, S\cup T} = d\paren{y, X \cup \partial C} > d\]
Therefore,
\[Y\cap \paren{S_d}^c \supseteq Y\cap \paren{S_d\cup T_d}^c = Y\cap \paren{X_d \cup \paren{\partial C}_{d}}^c \neq \varnothing\]
Applying the same argument to each $X_j$ and counting components of the complement yields the desired inequalities.

Otherwise, assume that $i\leq m-1.$ We will show that 
\[ M_{S\cup X, T }\paren{b,d} \geq M_{S, T}\paren{b,d}+M_{X, \partial C}\paren{b,d}\]
and the desired result will follow by induction.  Note that 
\[X_\epsilon\cap S_\epsilon \subseteq X_\epsilon \cap \paren{S_\epsilon \cup T_\epsilon} \subseteq \paren{\partial C}_{\epsilon}\]
for any $\epsilon>0.$ Consider the following commutative diagram of inclusion homomorphisms and Mayer-Vietoris sequences:
\begin{center}
\begin{tikzcd}
H_i\paren{X_b \cap S_b} \arrow[r] \arrow[d] & H_i\paren{X_b}\oplus H_i\paren{S_b} \arrow[r,"\alpha_{b}+\beta_b"] \arrow[d,"\phi \oplus \psi"] & H_i\paren{X_b \cup S_b} \arrow[d,"\zeta"]\\
0= H_i\paren{\paren{\partial C}_{d}} \arrow[r] & H_i\paren{X_d\cup \paren{\partial C}_d}\oplus H_i\paren{S_d \cup T_d} \arrow[r,"\alpha_{d}+\beta_d"] & H_i\paren{X_d\cup S_d \cup T_d} \\
\end{tikzcd}
\end{center}
Observe that $M_{X,\partial C}\paren{b,d}=\rank\,\phi,$ $M_{S,T}\paren{b,d}=\rank\,\psi,$ and  $M_{X\cup S,T}\paren{b,d}=\rank\,\zeta.$  It follows that 
\begin{align*}
M_{X\cup S,T}\paren{b,d}= &\\
&\;\; \rank \zeta \\
\geq &  \;\; \rank\,\paren{\alpha_d+\beta_d} \circ  \paren{\phi \oplus \psi}\\
= & \;\; \rank\,\paren{\phi \oplus \psi} &&\text{because $H_i\paren{\paren{\partial C}_{d}} =0$} \\
=  & \;\;\rank\,\phi + \rank\,\psi\\
= & \;\; M_{X,\partial C}\paren{b,d} +M_{S,T}\paren{b,d}\\
\geq & \;\; \sum_{j=1}^k M_{X_j,\partial C_j}\paren{b,d} &&\text{by induction}
\end{align*}
\end{proof}

\subsection{Occupancy Events}

If $B$ is a subset of an $m$-space, define the occupancy event
\[\delta\paren{B,\textbf{x}}=
\begin{cases} 
0 & \abs{\textbf{x}\cap B}=0\\
1 & \abs{\textbf{x}\cap B}>0
\end{cases}
\]

Also, if $\set{A_i}_{i=1}^r$ and $\set{B_j}_{j=1}^s,$ are collections of  subsets of $M$, let
\[\xi\paren{\textbf{x},\set{A_i},\set{B_j}}=\begin{cases}
 1 & \delta\paren{A_i,\textbf{x}}=0 \text { and }\delta\paren{B_j,\textbf{x}}=1\qquad \forall\, i,j\\
 0 & \text{otherwise}
\end{cases}\]

\begin{Lemma}
\label{lemma_weak_uniform}
Let $\mu$ be a locally bounded probability measure on an $m$-space $M.$ There exists a real number $V_0>0$ so for any $r,s\in\N$ there there exists a  real number $\gamma_0>0$ so that for any collections of disjoint, congruent cubes$\set{A_i^k}$ and $\set{B_j^k},$ for $i\in\set{1,\ldots,r},$ $j\in\set{1,\ldots,s},$ and $k\in\set{1,\ldots,n}$ (for a total of $\paren{r+s}n$ cubes) with volume

\[\vol\paren{A_i^k}=\vol\paren{B_j^k}=V_0/n \qquad \forall \, i,j,k\]
then
\[\mathbb{P}\paren{\sum_{k=1}^n \xi\paren{\textbf{x},\set{A_i^k},\set{B_j^k}} \geq s}\geq \gamma_0\]
for all sufficiently large $n.$ 
\end{Lemma}
\begin{proof}
The proof is nearly identical to that of Lemma 3 in~\cite{2018schweinhart_c}.
\end{proof}

\begin{Lemma}
\label{lemma_lower_1}
Let $0<b<d<1/6,$ and $V_0>0.$ There exists a $\lambda_0>0$ so that if $C\subset\R^m$ is an $m$-dimensional cube of width $R$ and $\lambda>\lambda_0,$  there exist disjoint, congruent cubes $\set{A_j}$ and $\set{B_k}$ of width $R\paren{V_0/\lambda}^{\frac{1}{m}}$ so that
\[\xi\paren{\textbf{x},\set{A_j},\set{B_k}} =1 \implies M_{\textbf{x},\partial C}\paren{R b,R d}>0 \]
\end{Lemma}
\begin{proof}

\begin{figure}
\centering
\includegraphics[width=.7\textwidth]{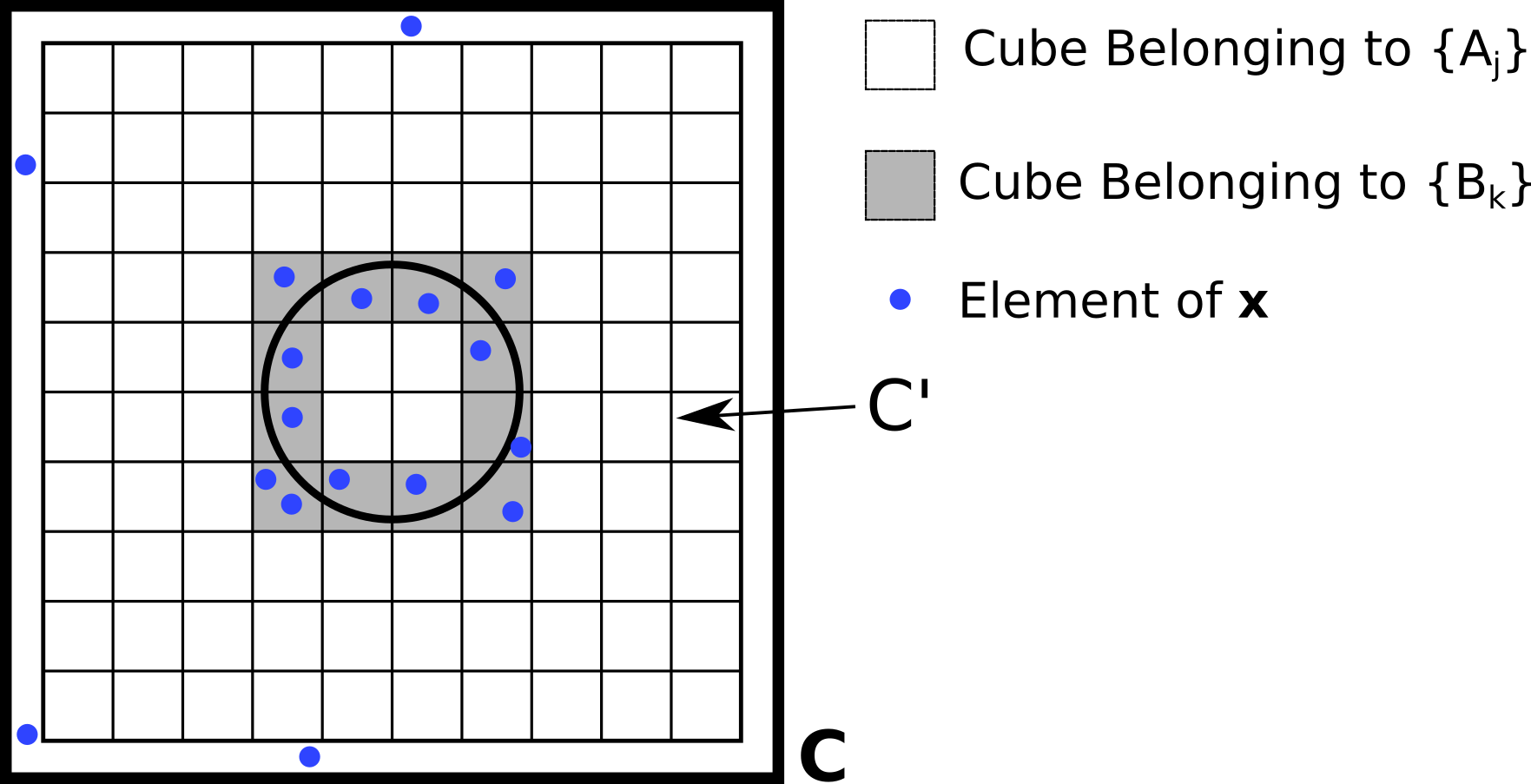}
\caption{The setup in the proof of Lemma~\ref{lemma_lower_1}.}
\label{fig:circlefig}
\end{figure}

We may assume without loss of generality that $R=1$ and $C$ is centered at the origin. Let $S^i\subset \R^m$ be an $i-$dimensional sphere of diameter $1/3$ centered at the origin; note that $\PH_i\paren{S^i}$ consists of a single interval $\paren{0,1/6}.$ 

Let $\kappa=\min\paren{b,1/6-d}$ and $\Delta_0=\kappa/\sqrt{m}.$ 
\[\lim_{\delta\rightarrow 0}\,\delta \floor{1/\delta}=1\]
so there is a real number $\Delta_1>0$ so that $1-\delta \floor{1/\delta}<\kappa$ for all $\delta<\Delta_1.$ Set
\[\lambda_0=\frac{V_0}{\min\paren{\Delta_0,\Delta_1}^m}\]

Choose $\lambda>\lambda_0,$ set $\delta=\paren{V_0/\lambda}^{\frac{1}{m}},$ and let $C'$ be the cube of width $\delta\floor{1/\delta}$ centered at the origin. Subdivide $C'$ into $\floor{1/\delta}^m$ sub-cubes of width $\delta.$ Call this collection of sub-cubes $\set{C_l}$ and let
\[\set{A_j}=\set{c\in\set{C_l}:S^i\cap c=\varnothing} \quad\text{and}\quad \set{B_k}=\set{c\in\set{C_l}:S^i\cap c\neq \varnothing}\]
See Figure~\ref{fig:circlefig} for an illustration. 

If $\textbf{x}\subset C$ and the event $\xi\paren{\textbf{x},\set{A_j},\set{B_k}}$ occurs , then
\[d_H\paren{\textbf{x}\cap C', S^i}<\kappa\]
where $d_H$ is the Hausdorff distance and we used the fact that the diagonal of an $m$-dimensional cube of width $\delta$ is $\delta\sqrt{m}.$  The stability of the bottleneck distance~\cite{2007cohensteiner} implies that $\PH_i\paren{\textbf{x}\cap C'}$ includes an interval $\paren{\hat{b},\hat{d}}$ so that
\[\hat{b}<\kappa\leq b <d \leq 1/3-\kappa < \hat{d}\]
In particular,
\[N_{\textbf{x}\cap C'}\paren{b,d}>0\]

By construction,
\[\frac{1}{2} d\paren{\textbf{x}\cap C',C\setminus C'} > \frac{1}{2}\paren{\frac{1}{6}\sqrt{m}\delta - d\paren{C,C'}} > \frac{1}{6} -\kappa \geq d\]
so the $\epsilon$-neighborhoods of $\textbf{x}\cap C'$ and $C\setminus C'$ are disjoint for all $\epsilon\leq d.$ It follows that the maps on homology induced by the inclusions $\paren{\textbf{x}\cap C'}_{\epsilon} \hookrightarrow \textbf{x}_{\epsilon}$ and $\textbf{x}_{\epsilon}\hookrightarrow \textbf{x}_{\epsilon}\cup \paren{\partial C}_\epsilon$ are injective for all $\epsilon\leq d.$ Therefore, $M_{\textbf{x},\partial C}\paren{b,d}>0,$ as desired.
\end{proof}

\subsection{Proof of the Lower Bound}
In the remainder, let $\mu$ be a locally bounded probability measure on an $m$-space $M,$ let $\set{x_j}_{j\in\N}$ be i.i.d. samples from $\mu,$ and let $\textbf{x}_n=\set{x_1,\ldots,x_n}.$ Also, let $C$ be as in Lemma~\ref{lemma_weak_uniform}, and rescale $\Delta_M$ if necessary so that $C$ is a unit cube. Finally, let $\lambda_0$ be as in Lemma~\ref{lemma_lower_1}.

\subsubsection{The Euclidean Case}
For clarity, we first consider the special case where $\phi_M$ is the identity map, and $\mu$ is a locally bounded probability measure on a compact Euclidean simplicial complex. The argument for the general case contains many of the same elements.

\begin{Lemma}
\label{lemma_lower_2}
Let $0<b_0<d_0<1/6.$ If $n_0>\lambda_0,$ there is a $\gamma_1>0$ so that 
\[\lim_{n\rightarrow\infty} \frac{1}{n} N_{\textbf{x}_n}\paren{\paren{\frac{n_0}{n}}^{\frac{1}{m}}b_0,\paren{\frac{n_0}{n}}^{\frac{1}{m}}d_0}> \gamma_1 \]
in probability.
\end{Lemma}
\begin{proof}

Let $V_0$ be as in Definition~\ref{lemma_weak_uniform}, and let $r=\abs{A_i}$ and $s=\abs{B_j},$ where $\set{A_i}$ and $\set{B_j}$ are as in the previous lemma. 

Assuming $n>n_0,$ let $\omega=\paren{\frac{n_0}{n}}^{\frac{1}{m}}.$ Subdivide $\R^m$ into cubes of width $\omega,$ and let $\set{D_l}_{l=1}^{K_n}$ be the cubes that are fully contained in $C.$ Note that 
\[K_n \coloneqq \abs{\set{D_l}}\approx n/n_0\]

By the previous lemma, there are collections of disjoint, congruent sub-cubes $\set{A^l_1,\ldots,A^l_r}$ and $\set{B^l_1,\ldots,B^l_s}$ of width $\omega \paren{V_0/n_0}^{\frac{1}{m}}$ contained inside each cube $D_l$ so that
\[\xi\paren{\textbf{x}_n,\set{A^l_i},\set{B^l_j}}=1\implies M_{\textbf{x}_n\cap D_l,\partial D_l}\paren{\omega b_0,\omega d_0}>0 \]

Note that
\begin{align*}
N_{\textbf{x}_n}\paren{\omega b_0,\omega d_0}\geq & \\
&\;\; \sum_{l=1}^{K_n} M_{\textbf{x}_n\cap D_l,\partial D_l}\paren{\omega b_0,\omega d_0} &&\text{by Lemma~\ref{lemma_additivity}}\\
\geq & \;\; \sum_{l=1}^{K_n} \xi\paren{\textbf{x}_n,\set{A^l_i},\set{B^l_j}}
\end{align*}

Let $\gamma_0$ be as in Lemma~\ref{lemma_weak_uniform} and $\gamma<\gamma_0/n_0.$ Set 
\[\delta=\frac{1+\gamma \frac{n_0}{\gamma_0}}{2} \quad\text{and}\quad \epsilon=\frac{1-\delta}{\delta}\] 
so $1/2<\delta<1$ and $0<\epsilon<1.$ Also, find a $N$ so that $K_n>\delta n/n_0$ for all $n>N.$ Note that
\begin{equation}
\label{eqn_binomialLemma}
\gamma n=\frac{\gamma_0 \delta n}{n_0}\paren{1-\frac{1-\delta}{\delta}}<\paren{1-\epsilon}\gamma_0 K_n
\end{equation}
for all $n>N.$ Therefore, if $n>N,$ 
\begin{align*}
\mathbb{P}\paren{ N_{\textbf{x}_n}\paren{\omega b_0,\omega d_0} > \gamma n}\geq & \\
&\;\; \mathbb{P}\paren{\sum_{l=1}^{K_n} \xi\paren{\textbf{x}_n,\set{A^l_i},\set{B^l_j}} > \gamma n}\\
\geq & \;\;\mathbb{P}\paren{B\paren{K_n,\gamma_0}>\gamma n} &&\text{by Lemma~\ref{lemma_weak_uniform}}\\
\geq & \;\; \mathbb{P}\paren{B\paren{K_n,\gamma_0}>\paren{1-\epsilon}\gamma_0 K_n} &&\text{by Equation~\ref{eqn_binomialLemma}}
\end{align*}
which converges to $1$ as $n\rightarrow\infty.$
\end{proof}

We can now prove the lower bound in the Euclidean setting:
\begin{Proposition}
\label{prop_lower_1}
There is a $\gamma'>0$ so that 

\[\lim_{n\rightarrow\infty} n^{-\frac{m-\alpha}{m}} E_\alpha^i\paren{\textbf{x}_n}\geq  \gamma' \]
in probability.
\end{Proposition}
\begin{proof}
Let $0<b<d<1/6,$ and let $n_0>\lambda_0$ and $\gamma_1$ be as before. Also, let $\omega=\paren{\frac{n_0}{n}}^{1/m}.$ We have that

\begin{align*}
\lim_{n\rightarrow\infty}  n^{-\frac{m-\alpha}{m}} E_\alpha^i\paren{\textbf{x}_n}\geq&\\
& \;\; \lim_{n\rightarrow\infty}   n^{-\frac{m-\alpha}{m}} \paren{\omega d- \omega b}^\alpha N_{\textbf{x}_n}\paren{\omega b,\omega d}\\
= & \;\; \lim_{n\rightarrow\infty} \frac{n_0^{\alpha/m}}{n}   \paren{d-b }^\alpha N_{\textbf{x}_n}\paren{\omega b,\omega d}\\
\geq &  \;\;  n_0^{\alpha/m} \paren{d-b}^\alpha \gamma_1  &&\text{by Lemma~\ref{lemma_lower_2}}\\
\coloneqq & \;\; \gamma'
\end{align*}
in probability.
\end{proof}

\subsubsection{The General Case}
Before proving the lower bound in our main theorem, we require an interleaving result for the persistent homology of images of bi-Lipschitz maps:

\begin{Lemma}
\label{lemma_interleave}
Let $M_0$ and $M_1$ be metric spaces and let $\psi:M_0\rightarrow M_1$ be $L$-bilipshitz. If $X\subset M_0$ and $0\leq b_0 < d_0$
\[N_X\paren{b_0/L,L d_0}\leq N_{\psi\paren{X}}\paren{b_0,d_0}\leq N_{X}\paren{L b_0,d_0/L}\] 
\end{Lemma}
\begin{proof}
Fix $i\in\N,$ and let $j_{\epsilon_0,\epsilon_1}:X_{\epsilon_0}\hookrightarrow X_{\epsilon_1}$ and  $k_{\epsilon_0,\epsilon_1}:\phi\paren{X}_{\epsilon_0}\hookrightarrow\phi\paren{X}_{\epsilon_1}$ denote the inclusion maps for $\epsilon_0\leq \epsilon_1.$ 

By the definition of a bi-Lipschitz map
\[\frac{1}{L} d_{M_0}\paren{x,y}\leq d_{M_1}\paren{\psi\paren{x},\psi\paren{y}}\leq L d_{M_0}\paren{x,y}\]
for all $x,y\in M_0.$ In particular, we have the following inclusions:
\[\psi\paren{X_{b_0/L}} \hookrightarrow \psi\paren{X}_{b_0}  \hookrightarrow \psi\paren{X}_{d_0} \hookrightarrow \psi\paren{X_{L d_0}}\]

It follows that the rank of map on homology induced by $i_{b_0/L,Ld_0}$ is less than or equal to the rank of the map induced by $j_{b_0,d_0}$ (where we have used that a bi-Lipschitz map is a homeomorphism). Therefore, 
\[N_{X}\paren{b_0/L,L d_0}\leq N_{\psi\paren{X}}\paren{b_0,d_0}\]

The argument for the other inequality is very similar.
\end{proof}

\begin{Proposition}
\label{prop_lower}
Let $\mu$ be a locally bounded probability measure on an $m$-space $M$ and $0\leq i<m.$ There is a $\gamma>0$ so that 
\[\lim_{n\rightarrow\infty} n^{-\frac{m-\alpha}{m}} E_\alpha^i\paren{x_1,\ldots x_n}> \gamma \]
in probability
\end{Proposition}

\begin{proof}
Let $L$ be the bi-Lipschitz constant of $\phi_M,$ and choose $b,d>0$ so that 
\[L^2 b < d <1/6\]
Set 
\[n_0=\max\paren{\paren{d/L-Lb}^{-m},n_0}\]
 so 
\begin{equation} 
\label{prop_lower_eqn_2}
n_0^{\frac{1}{m}}\paren{d/L-L b}\geq 1
\end{equation}

Let $\omega=\paren{\frac{n_0}{n}}^{\frac{1}{m}}$ and $\textbf{y}_n=\phi_m^{-1}\paren{\textbf{x}_n}.$ Our strategy is to bound $E_\alpha^i\paren{\textbf{x}_n}$ by applying Lemma~\ref{lemma_lower_2} to $\textbf{y}_n.$  

First, 
\begin{align*}
E_\alpha^i\paren{\textbf{x}_n} &\\
& \;\; \paren{\omega \paren{d/L-Lb}}^\alpha N_{\textbf{x}_n}\paren{\omega Lb, \omega d/L}\\
\geq & \;\;n^{-\alpha/m} N_{\textbf{x}_n}\paren{\omega Lb, \omega d/L} &&\text{by Equation~\ref{prop_lower_eqn_2}}\\
\geq & \;\; n^{-\alpha/m} N_{\textbf{y}_n}\paren{\omega b, \omega d}&&\text{by Lemma~\ref{lemma_interleave}}\\
= & \;\;  n^{-\alpha/m} N_{\textbf{y}_n}\paren{\omega b, \omega d}
\end{align*}

Therefore, 
\[\lim_{n\rightarrow\infty} n^{-\frac{m-\alpha}{m}} E_\alpha^i\paren{\textbf{x}_n} \geq \lim_{n\rightarrow\infty} \frac{1}{n}  N_{\textbf{y}_n}\paren{\omega b, \omega d}> \gamma_1\]
in probability, where $\gamma_1>0$ is as given in Lemma~\ref{lemma_lower_2}.

\end{proof}

\nocite{2018bauer}
\nocite{2017bobrowski}

\bibliographystyle{plain}
\bibliography{Bibliography}

\end{document}